\documentclass[letterpaper, oneside]{amsart}

\usepackage{layout}	

\usepackage[
]{geometry}

\usepackage[parfill]{parskip}   
\usepackage{amssymb}
\usepackage{amsthm}
\usepackage{amsmath}
\usepackage{cite}
\usepackage[all]{xy}
\usepackage{enumerate}
\usepackage{verbatim}
\usepackage{mathtools}
\usepackage{hyperref}
\usepackage{colonequals}
\usepackage{mathrsfs}
\usepackage{color}

\usepackage{tikz}
\usetikzlibrary{decorations.markings}

\usepackage{longtable}
\usepackage{multicol}
\usepackage{multirow}

\theoremstyle{plane}
\newtheorem{thm}{Theorem}[section]
\newtheorem{lem}[thm]{Lemma}

\newtheorem{prop}[thm]{Proposition}

\theoremstyle{definition}

\newtheorem*{rmk}{Remark}
\newtheorem*{ack}{Acknowledgement}

\newtheorem{blank}[thm]{}
\newtheorem{example}[thm]{Example}

\newcommand{\A}{\mathbb{A}}

\newcommand{\Z}{\mathbb{Z}}
\newcommand{\N}{\mathbb{N}}
\newcommand{\Q}{\mathbb{Q}}

\let\P\relax
\newcommand{\P}{\mathbb{P}}
\newcommand{\C}{\mathbb{C}}

\let\k\relax
\newcommand{\k}{\mathbf{k}}

\newcommand{\B}{\mathcal{B}}

\newcommand{\cP}{\mathcal{P}}

\let\S\relax
\newcommand{\S}{\mathcal{S}}

\newcommand{\g}{\mathfrak{g}}
\let\b\relax
\newcommand{\b}{\mathfrak{b}}
\newcommand{\gl}{\mathfrak{gl}}
\newcommand{\so}{\mathfrak{so}}
\let\sp\relax
\newcommand{\sp}{\mathfrak{sp}}

\let\l\relax
\newcommand{\l}{\ell}

\newcommand{\spr}{\mathbf{TSp}}

\newcommand{\ec}{\mathbf{EC}}

\newcommand{\br}[1]{\left\langle{#1}\right\rangle}
\newcommand{\brr}[1]{\left({#1}\right)}

\newcommand{\qlbar}{{\overline{\mathbb{Q}_\ell}}}

\newcommand{\floor}[1]{\left\lfloor#1\right\rfloor}

\DeclareMathOperator{\sgn}{\textup{sgn}}

\DeclareMathOperator{\im}{\textup{im}}
\DeclareMathOperator{\Ind}{\textup{Ind}}
\DeclareMathOperator{\Res}{\textup{Res}}
\DeclareMathOperator{\tA}{\tilde{\textit{A}}}
\DeclareMathOperator{\oA}{\overline{\textit{A}}}
\DeclareMathOperator{\tr}{\textup{tr}}

\DeclareMathOperator{\ch}{\textup{char}}

\title{Euler Characteristic of Springer fibers}
\author{Dongkwan Kim}
\address{Department of Mathematics\\
  Massachusetts Institute of Technology\\
  Cambridge, MA 02139-4307\\
  U.S.A.}
\email{sylvaner@math.mit.edu}
\date{\today}							

\begin{document}
\begin{abstract} For Weyl groups of classical types, we present formulas to calculate the restriction of Springer representations to a maximal parabolic subgroup of the same type. As a result, we give recursive formulas for Euler characteristics of Springer fibers for classical types. We also give tables of those for exceptional types.
\end{abstract}

\maketitle

\renewcommand\contentsname{}
\tableofcontents

\section{Introduction}
Suppose that we have a reductive group $G$ over an algebraic closed field $\k$ and its Lie algebra $\g$. Let $\B$ be the set of Borel subalgebras of $\g$ and for any $N \in \g$ define
$$\B_N \colonequals \{ \b \in \B \mid N \in \b \}$$
which we call the Springer fiber corresponding to $N$. It is currently one of the main objects in (geometric) representation theory.

Let $W$ be the Weyl group of $G$. \cite{spr} defined an action of $W$ on the cohomology of Springer fibers when $\k$ is an algebraic closure of some finite field, which is now called the Springer representation. It was reconstructed in more general setting by \cite{lu:green} using the theory of perverse sheaves, which differs from \cite{spr} by the sign character of $W$. 

In this paper we present formulas to calculate the restriction of Springer representations to a maximal parabolic subgroup $W' \subset W$ of the same type when $G$ is of classical type. In particular, we have recursive formulas for Euler characteristics of Springer fibers. Furthermore, we have recursive formulas to calculate the multiplicities in these representations of each irreducible representation of $A_N$, the component group of the stabilizer of $N$ by adjoint action of $G$ on $\g$. For $G$ of exceptional type, we give a table for such multiplicities in each Springer representation in Appendix \ref{sec:exceptional}.

In some special cases we also have such recursive formulas for the multiplicities of irreducible representations of $A_N$ in the cohomology of Springer fibers on each degree. As a result we present closed formulas of such multiplicities for ``two-row" cases.

Note that the Green functions of $G$ can be calculated using the Lusztig-Shoji algorithm (e.g. \cite{shoji}, \cite{lu:char5}), and it is also possible to obtain information of Springer representations from Green functions. But our method is more elementary and does not use orthogonality of Green functions, which is crucial for the Lusztig-Shoji algorithm.

\begin{ack} The author thanks George Lusztig for giving thoughtful comments, revising drafts of this paper, and encouraging the author to publish it.
\end{ack}

\section{Notations and Preliminaries}
\begin{blank}
Let $\k$ be an algebraically closed field. Throughout this paper we assume that $G$ is $GL_n$, $SO_{2n+1}$, $Sp_{2n}$, or $SO_{2n}$ over $\k$, except in Appendix \ref{sec:exceptional}. Here $n\geq 1$, unless $G=SO_{2n}$ in which case $n \geq 2$. We identify $G$ with the set of its $\k$-points $G(\k)$. We also assume that $\ch \k$ is good; there is no assumption on $\ch \k$ if $G=GL_n$ and $\ch \k \neq 2$ otherwise. If $G=GL_n$ (resp. $G=SO_{2n+1}, Sp_{2n}, SO_{2n}$) we define $V$ to be a $\k$-vector space of dimension $n$ (resp. $2n+1, 2n, 2n$), respectively, on which $G$ naturally acts. Also when $G= SO_{2n+1}$, (resp. $Sp_{2n}, SO_{2n},$) $V$ is equipped with a symmetric (resp. symplectic, symmetric) bilinear form $\br{\ , \ }$ which is invariant under the action of $G$, i.e. for any $v, w \in V$ and $g \in G$ we have $\br{gv, gw} = \br{v, w}.$
\end{blank}

\begin{blank}
Let $W$ be the Weyl group of $G$ and $S = \{s_1, \cdots, s_n\} \subset W$ be the set of simple reflections of $W$ such that $(W, S)$ is a Coxeter group. We choose $s_i$ such that the labeling corresponds to one of the following Dynkin diagrams.

\begin{center}
Type $A_n$:  \begin{tikzpicture}[scale=.7]
	\node at (0 cm,0) {$s_1$};
	\node at (2 cm,0) {$s_2$};
	\node at (4 cm,0) {$s_3$};
	\node at (6 cm,0) {$\cdots$};
	\node at (8 cm,0) {$s_n$};

    \draw[thick] (0 cm,0) circle (3 mm);
    \draw[thick] (2 cm,0) circle (3 mm);
    \draw[thick] (4 cm,0) circle (3 mm);
    \draw[thick] (8 cm,0) circle (3 mm);
    
    \draw[thick] (0.3, 0) -- (1.7,0);
    \draw[thick] (2.3, 0) -- (3.7,0);
    \draw[thick] (4.3, 0) -- (5.5,0);
    \draw[thick] (6.5, 0) -- (7.7,0);
  \end{tikzpicture}

Type $BC_n$:  \begin{tikzpicture}[scale=.7]
	\node at (0 cm,0) {$s_1$};
	\node at (2 cm,0) {$s_2$};
	\node at (4 cm,0) {$s_3$};
	\node at (6 cm,0) {$\cdots$};
	\node at (8 cm,0) {$s_n$};

    \draw[thick] (0 cm,0) circle (3 mm);
    \draw[thick] (2 cm,0) circle (3 mm);
    \draw[thick] (4 cm,0) circle (3 mm);
    \draw[thick] (8 cm,0) circle (3 mm);
    
    \draw[thick,double] (0.3, 0) -- (1.7,0);
    \draw[thick] (2.3, 0) -- (3.7,0);
    \draw[thick] (4.3, 0) -- (5.5,0);
    \draw[thick] (6.5, 0) -- (7.7,0);
  \end{tikzpicture}
  
Type $D_n$:  
\begin{tikzpicture}[baseline=(tes.base), scale=.7]
	\node at (0 cm,1cm)  {$s_1$};
	\draw[thick] (0 cm,1cm) circle (3 mm);
	\node at (0 cm,-1cm) {$s_2$};
	\draw[thick] (0 cm,-1cm) circle (3 mm);
	\node at (2 cm, 0) (tes) {$s_3$};
	\draw[thick] (2 cm, 0) circle (3 mm);
	\node at (4 cm, 0) {$\cdots$};	
	\node at (6 cm, 0) {$s_n$};
	\draw[thick] (6 cm, 0) circle (3 mm);

	\draw[thick] (1.8 cm, 2mm)  -- +(-15 mm, 8mm);
	\draw[thick] (1.8 cm,-2mm)  -- +(-15 mm,-8mm);
	\draw[thick] (2.3 cm,0)  -- +(12 mm,0);
	\draw[thick] (4.5 cm,0)  -- +(12 mm,0);
  \end{tikzpicture}
\end{center}
\end{blank}

\begin{blank}
Let $\g=\gl_n$ (resp. $\so_{2n+1}, \sp_{2n},$ $\so_{2n}$) be the Lie algebra of $G=GL_n$ (resp. $SO_{2n+1}$, $Sp_{2n},$ $SO_{2n}$). Then if $\g =\so_{2n+1}, \sp_{2n},$ or $\so_{2n}$ there is a natural Lie algebra action of $\g$ on $V$ which respects $\br{\ , \ }$, i.e. for any $v, w \in V$ and $N \in \g$ we have $\br{Nv, w}+\br{v, Nw} = 0$.
\end{blank}

\begin{blank}
Let $\B$ be the flag variety of $G$, i.e. the set of Borel subgroups of $G$, or equivalently, the set of Borel subalgebras of $\g$. If $G=GL_n$, then $\B$ is isomorphic to the variety of full flags in $V$. If $G=SO_{2n+1}$ or $Sp_{2n}$, then $\B$ is isomorphic to the variety of full isotropic flags in $V$. If $G=SO_{2n}$, then $\B$ is isomorphic to the variety of isotropic flags $0 \subsetneq V_1 \subsetneq \cdots \subsetneq V_{n-1}$ in $V$ where $\dim V_i = i$. For $N \in \g$, we write $\B_N$ to be the Springer fiber of $N$, i.e. the set of Borel subalgebras of $\g$ that contains $N \in \g$. 
\end{blank}

\begin{blank}
For $G=GL_n$ with $n\geq 1$, (resp. $SO_{2n+1}$ with $n\geq 1$, $Sp_{2n}$ with $n\geq 1$, $SO_{2n}$ with $n\geq 2$,) we define $G'=GL_{n-1}$ (resp. $SO_{2n-1}$, $Sp_{2n-2}, SO_{2n-2}$) and $\g' = \gl_{n-1}$ (resp. $\g'=\so_{2n-1}$, $\sp_{2n-2}, \so_{2n-2}$). Also we define $\B'$ to be the flag variety of $G'$ and $W'$ to be the Weyl group of $G'$. We regard $W'$ as a subgroup of $W$ generated by $S'=\{s_1, \cdots, s_{n-1}\} \subset S$, except when $G=SO_{4}$ in which case we regard $W' = \{id\}\subset W$. If $N' \in \g'$, let $ \B'_{N'}$ be the Springer fiber corresponding to $N'$ with respect to $G'$.
\end{blank}

\begin{blank}
For a variety $X$, let $H^*(X):=\sum_{i \in \N} (-1)^iH^i(X, \qlbar)$ be the alternating sum of $\l$-adic cohomology of $X$ in the Grothendieck group of vector spaces. If $\k = \C$, then by comparison theorem it is equivalent to the alternating sum of complex cohomology of $X^{an}$. Similarly, let $H^*_c(X):=\sum_{i \in \N} (-1)^iH^i_c(X, \qlbar)$ be the alternating sum of $\l$-adic cohomology with compact support of $X$.
\end{blank}

\begin{blank}
Let $\lambda$ be a partition. We write $\lambda \vdash n$ or $|\lambda|=n$ if $\lambda$ is a partition of $n$. We describe each part of $\lambda$ by writing $\lambda = (\lambda_1, \lambda_2, \cdots, \lambda_r) \vdash n$ where $\lambda_1\geq \lambda_2 \geq \cdots \geq \lambda_r >0$ and $|\lambda| = n$. Or we also write $\lambda = (1^{r_1}2^{r_2}\cdots)$ which means that $\lambda$ consists of $r_1$ parts of 1, $r_2$ parts of 2, and so on.
\end{blank}

\begin{blank} \label{note:partition}
Let $\lambda = (1^{r_1}2^{r_2}3^{r_3}\cdots)$. Then for $i \geq 1$ we define $\lambda^i, \lambda^{h,i}, \lambda^{v,i}$ as follows.
\begin{enumerate}
\item If $r_i \geq 1$, define $\lambda^i \colonequals (1^{r_1}2^{r_2} \cdots (i-1)^{r_{i-1}+1}i^{r_i -1} \cdots) \vdash |\lambda|-1$. This corresponds to removing a box from the Young diagram of $\lambda$.
\item If $r_i \geq 1$ and $i \geq 2$, we let $\lambda^{h,i}\colonequals(1^{r_1}2^{r_2}\cdots (i-2)^{r_{i-2}+1} (i-1)^{r_{i-1}}i^{r_i-1}\cdots) \vdash |\lambda|-2$. Here the superscript $``h"$ stands for ``horizontal"; $\lambda^{h,i}$ is obtained by removing a horizontal domino from the Young diagram of $\lambda$ and rearranging rows if necessary.
\item If $r_i \geq 2$, we let $\lambda^{v,i}\colonequals(1^{r_1}2^{r_2}\cdots  (i-1)^{r_{i-1}+2}i^{r_i-2}\cdots)\vdash |\lambda|-2$. Here the superscript $``v"$ stands for ``vertical"; $\lambda^{v,i}$ is obtained by removing a vertical domino from the Young diagram of $\lambda$. 
\end{enumerate}
\end{blank}

\begin{blank}
We recall the correspondence between nilpotent adjoint orbits in $\g$ under the adjoint action of $G$ and partitions. If $G=GL_n$ then such orbits are parametrized by partitions of $n$. This correspondence is given by taking the sizes of Jordan blocks of any element in a nilpotent orbit regarded as an endomorphism on $V$. Likewise, if $G=SO_{2n+1}$, then nilpotent adjoint orbits in $\g$ are parametrized by $\lambda=(1^{r_1} 2^{r_2} 3^{r_3} \cdots) \vdash 2n+1$ such that $r_{2i} \equiv 0\pmod{2}$ for $i \geq 1.$ If $G=Sp_{2n}$, then nilpotent adjoint orbits in $\g$ are parametrized by $\lambda=(1^{r_1} 2^{r_2} 3^{r_3} \cdots) \vdash 2n$ such that $r_{2i-1} \equiv 0\pmod{2}$ for $i \geq 1$. In these cases we write $N_\lambda \in \g$ to be such a nilpotent element corresponding to $\lambda \vdash n$. It is well-defined up to adjoint action by $G$.
\end{blank}



\begin{blank}
If $G = SO_{2n}$, then it is almost the same as the case $G=SO_{2n+1}$, so if $\lambda=(1^{r_1} 2^{r_2} 3^{r_3} \cdots) \vdash 2n$ is a partition of the sizes of Jordan blocks of some nilpotent element in $\g$, then  $r_{2i} \equiv 0\pmod{2}$ for $i \geq 1.$ However this correspondence is no longer one-to-one since a partition consisting of even parts with even multiplicities, which we call \emph{very even}, corresponds to two adjoint nilpotent orbits in $\g$. Thus if $\lambda \vdash 2n$ is not very even, we write $N_\lambda \in \g$ to be such a nilpotent element corresponding to $\lambda$ which is again well-defined up to adjoint action by $G$. If $\lambda \vdash 2n$ is very even, then we write $N_{\lambda+}, N_{\lambda-}$ to distinguish two such nilpotent elements corresponding to $\lambda$ in different adjoint orbits. If there is no ambiguity or need to differentiate $N_{\lambda+}$ and $N_{\lambda-}$, we still write $N_{\lambda}$.
\end{blank}

\begin{blank} \label{note:cent}
Let $\tilde{G} = GL_n$ (resp. $O_{2n+1}, Sp_{2n}, O_{2n}$)  if $G=GL_n$ (resp. $SO_{2n+1}, Sp_{2n}, SO_{2n}$). For a nilpotent $N \in \g$, we define $\tA_N$ to be the component group of the stabilizer of $N$ in $\tilde{G}$. Also for a partition $\lambda$, define $\tA_\lambda \colonequals \tA_{N_\lambda}$. If $G=SO_{2n}$ and $\lambda$ is very even, then as $\tA_{N_{\lambda+}} = \tA_{N_{\lambda-}} = \{id\}$, we define $\tA_{\lambda} \colonequals \{id\}$. If $\tilde{G}=GL_n$, any $\tA_{N}$ is trivial. Otherwise if $\tilde{G} = Sp_{2n}$ (resp. $O_{2n+1}$ or $O_{2n}$) and if a partition $\lambda = (1^{r_1} 2^{r_2} \cdots)$ corresponds to a nilpotent element in $\g$, then $\tA_\lambda$ is a product of $\Z/2$ generated by $z_i$ for each even $i$ (resp. odd $i$) such that $r_i>0$. (Here we adopt the convention that $z_i = id$ if $i$ does not satisfy the aforementioned condition.) Likewise, we define $A_N$ to be the component group of the stabilizer of $N$ in $G$ and $A_{\lambda} \colonequals A_{N_\lambda}$, again even when $\lambda$ is very even (in which case $A_\lambda$ is trivial). Then $A_N$ can be considered as a subgroup of $\tA_N$. If $G=GL_n$ or $Sp_{2n}$, $A_N = \tA_N$. Otherwise, for $\lambda = (1^{r_1}2^{r_2}\cdots)$ $A_\lambda$ is the subgroup of $\tA_\lambda$ generated by $\{z_i z_j  \mid  i, j \textup{ odd, } r_i, r_j >0\}$. 
\end{blank}

\begin{blank}
For a partition $\lambda$, define $H^i(\lambda) \colonequals H^i(\B_{N_\lambda})$ and $H^i(\lambda+) \colonequals H^i(\B_{N_{\lambda+}}), H^i(\lambda-) \colonequals H^i(\B_{N_{\lambda-}})$ if $G=SO_{2n}$ and $\lambda$ is very even.
We regard $H^i(\B_N)$ as $W$-modules using Springer theory, adopting the definition of \cite{lu:green}. 
We also consider the action of $\tA_N$ on $H^*(\B_N)$ which is induced from the action of $\tA_N$ on $\B_N$. Note that the action of $W$ and $A_N \subset \tA_N$ commute \cite[6.1]{spr}. We denote by $\spr(\lambda)$ the character of $H^*(\lambda)$ as a $W\times A_{\lambda}$-module. If $G=SO_{2n}$ and $\lambda$ is very even, we similarly define $\spr(\lambda+)$ and $\spr(\lambda-)$. Note that this definition does not depend on the base field $\k$ insofar as $\ch \k$ is good. Define $\ec(\lambda) \colonequals \dim H^*(\B_{N_\lambda})$ to be the $\ell$-adic Euler (or Euler-Poincar\'e) characteristic of $\B_{N_\lambda}$. This is well-defined even when $\lambda$ is very even as $\B_{N_{\lambda+}}$ and $\B_{N_{\lambda-}}$ are isomorphic. We also define $h^k(\lambda):=\dim H^k(\lambda)$ be the $k$-th Betti number of $\B_{N_\lambda}$.
\end{blank}


\section{Main theorem (weak form)}
The (weak version of the) main theorem in this paper is as follows.
\begin{thm}\label{thm:main} Recall the notations in \ref{note:partition}.
\begin{enumerate}[i)]
\item Let $G=GL_n$ and $\lambda=(1^{r_1}2^{r_2}3^{r_3}\cdots)$ be a partition corresponding to a nilpotent adjoint orbit of $\g$. Then we have
$$\ec(\lambda) = \sum_{r_i \geq 1} r_i \ec(\lambda^i).$$
\item Let $G=SO_{2n+1}$, $Sp_{2n}$, or $SO_{2n}$ and $\lambda=(1^{r_1}2^{r_2}3^{r_3}\cdots)$ be a partition corresponding to a nilpotent adjoint orbit of $\g$. Then we have
$$\ec(\lambda) = \sum_{i\geq 2,\ r_i \textup{ odd}}\ec(\lambda^{h,i}) + \sum_{i\geq 1,\ r_i \geq 2} 2\floor{\frac{r_i}{2}}\ec(\lambda^{v,i}).$$
Here $\floor{\frac{r_i}{2}}$ is the greatest integer that is not bigger than $\frac{r_i}{2}$.
\end{enumerate}
\end{thm}
Indeed, this can be deduced from more general statements, i.e. Theorem \ref{thm:mainA}, Theorem \ref{thm:bettiA}, Theorem \ref{thm:mainBCD}, and Theorem \ref{thm:bettiBCD}, which provide isomorphisms of $W'\times A_\lambda$-modules. In subsequent sections we prove such generalizations.

%
%
%
%

\section{Type $A$}
For $G=GL_n$, the following theorem generalizes Theorem \ref{thm:main}.
\begin{thm} \label{thm:mainA} Let $G=GL_n$ and $\lambda=(1^{r_1}2^{r_2}3^{r_3}\cdots)\vdash n$. Then we have 
$$\Res_{W'}^W \spr(\lambda) = \sum_{r_i \geq 1} r_i \spr(\lambda^i)$$
as characters of $W'$.
\end{thm}

\begin{rmk} This theorem can also be proved by a combinatorial method: see Appendix \ref{app:comA}.
\end{rmk}

\begin{proof} The method we adopt here using certain geometric properties of Springer fibers is well-known, cf. \cite{srini}, \cite{shimomura}, \cite{spaltenstein}, \cite{shoji}, \cite{leeuwen}, etc. But here we give a proof for the sake of completeness. 

Let $N = N_\lambda \in \g$ be a nilpotent element corresponding to $\lambda =(1^{r_1}2^{r_2} \cdots) \vdash n$ and $\cP$ be a variety of lines in $V$. Then we have a natural surjective morphism $\pi: \B \rightarrow \cP$ which sends $[0=V_0 \subset V_1 \subset \cdots \subset V_{n-1} \subset V_n =V]$ to $V_1$. This restricts to $\pi : \B_N \rightarrow \cP_N$ where $\cP_N$ is the set of lines annihilated by $N$. It is easy to show that $\pi : \B_N \rightarrow \cP_N$ is surjective.

Note that $\P(\ker N) = \cP_N$. We have filtration of $\ker N$
$$\ker N = W_0 \supset W_1 \supset W_2 \supset \cdots$$
where $W_i = \ker N \cap \im N^i$. Now suppose $W_{i-1}/W_{i}$ is nonzero. Then the action of $N$ on $W_{i-1}/W_{i}$ has Jordan blocks of size $i$, i.e. the restriction of $N$ onto $W_{i-1}/W_i$ corresponds to a (rectangular) partition $(i^{r_i})$. 

Let $\eta: \P(W_{i-1}) - \P(W_i) \rightarrow \P(W_{i-1}/W_i)$ be a canonical affine bundle with fiber $W_i$. We stratify $\P(W_{i-1}/W_i)$ into affine spaces, i.e. $\P(W_{i-1}/W_i) = \bigsqcup_{j=0}^{r_i-1} Y_j$ where $Y_j \simeq \A^j$. Also we let $\tilde{Y}_j \colonequals \eta^{-1}(Y_j)$ and $X_j \colonequals \pi^{-1}(\tilde{Y}_j) = (\eta\circ\pi)^{-1}(Y_j)$. Then we can choose this stratification so that $\pi : X_j \rightarrow \tilde{Y}_j$ is a locally trivial bundle with fiber $\B'_{N_{\lambda^i}}$, where $\lambda^i$ is defined in \ref{note:partition}. (ref. \cite{shimomura}) Thus we have
$$H^k_c(X_j) \simeq \bigoplus_{k_1+k_2=k}H^{k_1}_c(\tilde{Y}_j) \otimes H^{k_2}(\lambda^i) \simeq H^{k-2\dim W_i - 2j}(\lambda^i)$$
as a vector space. Indeed, more is true: first we recall \cite[Theorem 5.1]{hotta-shimomura}.
\begin{lem}\label{lem:hs} The Leray sheaves $R^j\pi_! \qlbar$ on $\cP_N$ have structures of $W'$-modules such that
\begin{enumerate}[(a)]
\item for any $x \in \cP_N$, $(R^j\pi_! \qlbar)_x \simeq H^j(\B'_{N'})$ as $W'$-modules where $H^j(\B'_{N'})$ is equipped with the Springer representation of $W'$ and where $N'$ is the image of $N$ under the canonical quotient morphism from the maximal parabolic subalgebra corresponding to $x$ (which contains $N$ by definition of $\cP_N$) to its Levi factor.
\item $H^k(\cP_N, R^j\pi_!\qlbar) \Rightarrow H^{j+k}(\B_N)$ is a spectral sequence of $W'$-modules, where the action of $W'$ on $H^{j+k}(\B_N)$ is the restriction of the Springer representation of $W$.
\end{enumerate}
\end{lem}
By the first part of Lemma \ref{lem:hs} we have the following result.
\begin{lem} \label{lem:WisoA}There exists a natural $W'$-module structure on $H^k_c(X_j)$ such that
$$H^k_c(X_j) \simeq H^{k-2\dim W_i - 2j}(\lambda^i)$$
as $W'$-modules. Here the action of $W'$ on $H^{k-2\dim W_i - 2j}(\lambda^i)$ is given by the Springer representation corresponding to $G'$.
\end{lem}

Now we consider the long exact sequences of the cohomology with compact support of $\B_N$ corresponding to the stratification by such $X_j$'s. By Lemma \ref{lem:WisoA}, these long exact sequences are defined in the category of $W'$-modules. Thus it follows that $H^*(\B_N)$ is the sum of such $H^*_c(X_j)$'s. Furthermore, the second part of Lemma \ref{lem:hs} and its proof implies that the $W'$-module structure on $H^*(\B_N)$, as a sum of such $H^*_c(X_j)$, coincides with the restriction to $W'$ of the Springer representation of $W$ on $H^*(\B_N)$. 

In sum, we have
$$\Res_{W'}^W \spr(\lambda) = \sum_{r_i\geq 1} r_i \spr(\lambda^i)$$
as characters of $W'$. But this is what we want to prove.
\end{proof}

Indeed, we may proceed further; since Springer fibers have vanishing odd cohomology by \cite{c-l-p}, it is also true for $X_j$. Thus each long exact sequence considered above splits into short exact sequences in even degrees. Therefore, by keeping track of degrees of each short exact sequence, we have the following theorem which generalizes Theorem \ref{thm:mainA}.

\begin{thm}\label{thm:bettiA} For $\lambda = (\lambda_1, \lambda_2, \cdots, \lambda_r)$ and any $k \in \Z$, we have
$$\Res_{W'}^W H^k(\lambda) = \bigoplus_{i=1}^r H^{k-2i+2}(\lambda^{\lambda_i})$$
as $W'$-modules. (See \ref{note:partition} for the definition of $\lambda^{\lambda_i}$.)
\end{thm}

\begin{rmk} If we evaluate the character of each side at $id \in W'$, then it gives \cite[Proposition 4.5]{fresse}. His method is combinatorial, counting the number of ``row-standard" tableaux of certain Young diagrams.
\end{rmk}

\begin{example} Suppose $\k$ is an algebraic closure of a finite field. For $w\in W$, we define $Q_{\lambda, w}(q) = \sum_{k \in \N} q^k\tr(w, H^{2k}(\lambda))$ to be the Green function associated with $w$ and the nilpotent element $N_{\lambda}\in\g$ corresponding to $\lambda$. Then Theorem \ref{thm:bettiA} implies that for $w\in W' \subset W$ we have
$$Q_{\lambda, w}(q) = \sum_{i=1}^{r}Q_{\lambda^{\lambda_i}, w}(q)q^{i-1}$$
where $Q_{\lambda^{\lambda_i}, w}$ is defined similarly to $Q_{\lambda, w}$. For example, if $G=GL_4$, $w = (1 2 3) \in W' \subset W$, then
\begin{align*}
Q_{(1,1,1,1),w}(q) &= (q^3+q^2+q+1)Q_{(1,1,1),w}(q), 
\\Q_{(2,1,1),w}(q) &= Q_{(1,1,1),w}(q) + (q^2+q)Q_{(2,1),w}(q),
\\Q_{(2,2),w}(q) &=(q+1)Q_{(2,1),w}(q),
\\Q_{(3,1),w}(q) &=Q_{(2,1),w}(q)+qQ_{(3),w},
\\Q_{(4),w}(q) &=Q_{(3),w}(q),
\end{align*}
which are true as
\begin{align*}
&Q_{(1,1,1,1),w}(q) =q^6 - q^4 - q^2 + 1, &&Q_{(2,1,1),w}(q)= Q_{(2,2),w}(q)=-q^2 + 1, 
\\&Q_{(3,1),w}(q)=Q_{(4),w}(q)=1, && Q_{(1,1,1),w}(q) =q^3 - q^2 - q + 1,
\\&Q_{(2,1),w}(q)=-q+1, &&Q_{(3),w}(q)=1. 
\end{align*}
\end{example}

\section{Type $B, C$ and $D$} \label{sec:BCD}
Now we assume that $G=SO_{2n+1}, Sp_{2n},$ or $SO_{2n}$. We follow the argument in the previous section with some modifications. Let $N = N_\lambda \in \g$ be a nilpotent element corresponding to $\lambda=(1^{r_1}2^{r_2} \cdots)$ and $\cP$ be a variety of isotropic lines in $V$. Then we have a natural surjective morphism $\pi: \B \rightarrow \cP$ defined by 
\begin{alignat*}{3}
&\textup{if } G=SO_{2n+1}, Sp_{2n}, \qquad &&[0=V_0 \subset V_1 \subset \cdots \subset V_{n-1} \subset V_n] &&\mapsto V_1, \textup{ and}
\\&\textup{if } G=SO_{2n}, \qquad &&[0=V_0 \subset V_1 \subset \cdots \subset V_{n-1}] &&\mapsto V_1.
\end{alignat*}
This restricts to $\pi : \B_N \rightarrow \cP_N$ where $\cP_N$ is the set of isotropic lines annihilated by $N$. It is easy to show that $\pi : \B_N \rightarrow \cP_N$ is surjective.

\cite{srini} defined a stratification on $\cP_N$ such that $\pi$ is locally trivial on each stratum, which is revisited in \cite{shoji2} and \cite{shoji}. We recall their results as follows. Start with a filtration
$$\ker N = W_0 \supset W_1 \supset W_2 \supset \cdots$$
where $W_i = \ker N \cap \im N^i$. Now suppose $W_{i-1}/W_{i}$ is nonzero. Then the action of $N$ on $W_{i-1}/W_{i}$ has Jordan blocks of size $i$, i.e. the restriction of $N$ onto $W_{i-1}/W_i$ corresponds to a (rectangular) partition $(i^{r_i})$.

\cite{srini} also defined a (non-degenerate) bilinear form $\brr{\ , \ }$ on $W_{i-1}/W_i$ which is symmetric if $i$ is odd (resp. even) and $G=SO_{2n+1}, SO_{2n}$ (resp. $G=Sp_{2n}$). Otherwise it is symplectic. If $\brr{\ ,\ }$ is symmetric, then the set of isotropic lines in $\P(W_{i-1}/W_i)$ forms a quadric hypersurface, which is nonsingular if $\dim W_{i-1}/W_i \geq 3$, a union of two points if $\dim W_{i_1}/ W_i = 2$, and empty if $\dim W_{i_1}/W_i=1$. If $\brr{\ , \ }$ is symplectic, then any $x \in \P(W_{i-1}/W_i)$ is isotropic.

There is a canonical affine bundle $\eta: \P(W_{i-1}) - \P(W_i) \rightarrow \P(W_{i-1}/W_{i})$ with fiber isomorphic to $W_i$. Now we define $Y$ or $Y_j$ to be one of the strata of $\P(W_{i-1}/W_i)$ in each case below, following argument in \cite{shoji}, and let $\tilde{Y} \colonequals \eta^{-1}(Y)$, $\tilde{Y}_j \colonequals \eta^{-1}(Y_j)$. (See \ref{note:partition} for the definition of $\lambda^{h,i}$ and $\lambda^{v,i}$.)
\begin{enumerate}[\textup{Case} I.]
\item Suppose $\brr{\ , \ }$ is symmetric. Let $Q$ be the set of isotropic lines with respect to $\brr{\ , \ }$ in $\P(W_{i-1}/W_i)$ and $C \colonequals \P(W_{i-1}/W_i) -Q$. There is a stratification
\begin{gather*}
Q= Q_0 \supset Q_1 \supset \cdots \supset Q_{m-1} \supset Q_{m+1} \supset \cdots \supset Q_{r_i-1} \supset Q_{r_i} = \emptyset
\\C= C_0 \supset C_1 \supset \cdots \supset C_{r_i-m-1} \supset C_{r_i-m} = \emptyset
\end{gather*}
defined in \cite{srini} or \cite{shoji}, where $m = \floor{\frac{r_i}{2}}$.
\begin{enumerate}
\item[$(a_1)$] $Y_j=Q_{j-1} - Q_{j}$ for $j \neq m, m+1$, or $Y_m = Q_{m-1} - Q_{m+1}$ when $r_i$ is odd. Then the fiber of $\pi$ at any point in $\tilde{Y}_j$ is isomorphic to $\B'_{N_{\lambda^{v,i}}}$. Also we have
$$Y_j \simeq \A^{r_i-j-1}\quad \textup{ if } 1 \leq j \leq m, \qquad Y_j \simeq \A^{r_i-j} \quad \textup{ if } m+2\leq j \leq r_i.$$
\item[$(a_2)$] $Y=Q_{m-1}-Q_{m+1}$ when $r_i$ is even. Then the fiber of $\pi$ at any point in $\tilde{Y}$ is isomorphic to $\B'_{N_{\lambda^{v,i}}}$. Also $Y \simeq \A^{r_i-m-1} \sqcup \A^{r_i-m-1} = \A^{m-1} \sqcup \A^{m-1}$.
\item[$(b_1)$] $Y_j=C_{j-1} - C_{j}$ for $j \neq m+1$. Then the fiber of $\pi$ at any point in $\tilde{Y}_j$ is isomorphic to $\B'_{N_{\lambda^{h,i}}}$. Also we have $Y_j \simeq \A^{r_i-j}-\A^{r_i-j-1}$.
\item[$(b_2)$] $Y=C_{r_i-m-1}=C_m$ when $r_i$ is odd. Then the fiber of $\pi$ at any point in $\tilde{Y}$ is isomorphic to $\B'_{N_{\lambda^{h,i}}}$. Also $Y \simeq \A^m$.
\end{enumerate}
\item Suppose $\brr{\ , \ }$ is symplectic and stratify $\P(W_{i-1}/W_i)$ with respect to some symplectic basis, say 
$$\P(W_{i-1}/W_i) = Z_0 \supset Z_1 \supset \cdots \supset Z_{r_i-1} \supset Z_{r_i}=\emptyset.$$
Let $Y_j = Z_{j-1}-Z_j$ be one of the strata. Then the fiber of $\pi$ at any point in $\tilde{Y}_j$ is isomorphic to $\B'_{N_{\lambda^{v,i}}}$. Also $Y_j \simeq \A^{r_i-j}$.
\end{enumerate}

Let $X\colonequals \pi^{-1}(\tilde{Y}) = (\eta\circ \pi)^{-1}(Y)$ and $X_j \colonequals\pi^{-1}(\tilde{Y}_j) =  (\eta\circ \pi)^{-1}(Y_j)$. Then we have the following theorem. (Here we do not differentiate $N_{\lambda+}$ and $N_{\lambda-}$ even when $G=SO_{2n}$ and $\lambda$ is very even, as they make no difference in the following statement. It is similar when $\lambda^{v,i}$ is very even. $\lambda^{h,i}$ cannot be very even in any case.) Note that $\lambda^{h,i}$ and $\lambda^{v,i}$ correspond to some nilpotent elements in $\g'$.

\begin{lem} \label{lem:ANiso} Let $z_a$ (resp. $z'_a$) be the generators of $\tA_{\lambda}$ (resp. $\tA_{\lambda^{h,i}}$ or $\tA_{\lambda^{v,i}}$) following the notations of \ref{note:cent}. Then we have following isomorphisms of $\tA_{\lambda}$-modules. 
\begin{enumerate}
\item[\textup{Case I}.]
\begin{enumerate}
\item[$(a_1)$] $H^k_c(X_j) \simeq H^{k-2\dim W_i - 2\dim Y_j}(\lambda^{v,i}).$
\item[$(a_2)$] $H^k_c(X) \simeq H^{k-2\dim W_i - 2\dim Y}(\lambda^{v,i}) \oplus H^{k-2\dim W_i - 2\dim Y}(\lambda^{v,i}).$
\item[$(b_1)$] If $k$ is even, $H^k_c(X_j) \simeq H^{k-2\dim W_i - 2\dim Y_j}(\lambda^{h,i})^\tau$.

If $k$ is odd,  $H^k_c(X_j) \simeq H^{k-2\dim W_i - 2\dim Y_j+1}(\lambda^{h,i})^\tau.$
\item[$(b_2)$] $H^k_c(X) \simeq H^{k-2\dim W_i - 2\dim Y}(\lambda^{h,i}).$
\end{enumerate}
\item[\textup{Case II}.] $H^k_c(X_j) \simeq H^{k-2\dim W_i - 2\dim Y_j}(\lambda^{v,i}).$
\end{enumerate}
Here $\tau=z_i' z_{i-2}' \in \tA_{\lambda^{h,i}}$. The action of $z_a \in \tA_{\lambda}$ on the right hand side is defined as follows. 
\begin{enumerate}
\item[\textup{Case I}.]
\begin{enumerate}
\item[$(a_1)$] $z_a$ acts by $z'_a$.
\item[$(a_2)$] $z_a$ acts by $z'_a$ for $a \neq i$, and $z_i$ permutes two summands.
\item[$(b_1)$] $z_a$ acts by $z'_a \mod \tau$ unless $r_i$ is even, $j=r_i/2$, $a=i$, and $k$ is odd, in which case $z_i$ acts by $-z'_i \mod \tau$, i.e. $v \mapsto -z'_i(v)$. (Note that the action of $z_i'$ and $z_{i-2}'$ on the right hand side are the same.)
\item[$(b_2)$] $z_a$ acts by $z'_a$ for $a \neq i$, and $z_i$ acts by $z'_{i-2}$.
\end{enumerate}
\item[\textup{Case II}.] $z_a$ acts by $z'_a$.
\end{enumerate}
\end{lem}
\begin{proof} This is exactly \cite[Proposition 2.4]{shoji} with some correction. (See also \cite[Lemma 3.3.1]{leeuwen}.) 
His formula has an error on the description of the action of $\tA_{\lambda}$ in Case I.$(b_1)$; it differs by a sign from our description in some special case. This sign comes from that the action of $z_i$ on $Y_j$ induces -1 on $H^{2\dim Y_j-1}_c(Y_j)$, which is equivalent to the reciprocal map $z\mapsto 1/z$ on $\A^1 -\{0\}$.
\end{proof}

Indeed, we have a similar result of Lemma \ref{lem:ANiso} for $W'$-modules, using argument in the previous section (mainly based on \cite[Theorem 5.1]{hotta-shimomura} and its proof).
Thus for any $Z \subset \cP_N$, $H^i(\pi^{-1}(Z))$ has a natural $W'$-module structure which comes from the Springer representations corresponding to $G'$. Now together with Lemma \ref{lem:ANiso} we have the following. (As before we do not differentiate $N_{\lambda+}$ and $N_{\lambda-}$ even when $G=SO_{2n}$ and $\lambda$ is very even. However, we need to be careful when $\lambda^{v,i}$ is very even.)
\begin{prop} \label{prop:Wiso} There is a natural $W'$-module structure on $H^k_c(X)$ and $H^k_c(X_j)$,  such that we have the following isomorphisms of $W' \times A_\lambda$-modules. 
\begin{enumerate}
\item[\textup{Case I}.]
\begin{enumerate}
\item[$(a_1)$] $H^k_c(X_j) \simeq H^{k-2\dim W_i - 2\dim Y_j}(\lambda^{v,i}).$
\item[$(a_2)$] If $G\neq SO_{2n}$ or $\lambda^{v,i}$ is not very even, then 
$$H^k_c(X) \simeq H^{k-2\dim W_i - 2\dim Y}(\lambda^{v,i}) \oplus H^{k-2\dim W_i - 2\dim Y}(\lambda^{v,i}).$$

If $G=SO_{2n}$ and $\lambda^{v,i}$ is very even, then 
$$H^k_c(X) \simeq H^{k-2\dim W_i - 2\dim Y}(\lambda^{v,i}+) \oplus H^{k-2\dim W_i - 2\dim Y}(\lambda^{v,i}-).$$
\item[$(b_1)$]  If $k$ is even, $H^k_c(X_j) \simeq H^{k-2\dim W_i - 2\dim Y_j}(\lambda^{h,i})^\tau$.

If $k$ is odd,  $H^k_c(X_j) \simeq H^{k-2\dim W_i - 2\dim Y_j+1}(\lambda^{h,i})^\tau.$
\item[$(b_2)$] $H^k_c(X) \simeq H^{k-2\dim W_i - 2\dim Y}(\lambda^{h,i}).$
\end{enumerate}
\item[\textup{Case II}.] $H^k_c(X_j) \simeq H^{k-2\dim W_i - 2\dim Y_j}(\lambda^{v,i}).$
\end{enumerate}
Here $\tau=z_i' z_{i-2}' \in \tA_{\lambda^{h,i}}$. The action of $A_\lambda$ is the restriction of that of $\tA_\lambda$ described in Lemma \ref{lem:ANiso}.
\end{prop}
\begin{proof} It can be proved similarly to Lemma \ref{lem:ANiso}, using Lemma \ref{lem:hs}. Also this is similar to \cite[Proposition 2.8]{shoji2} or \cite[Lemma 3.4]{shoji}, which only deal with cohomology on the top degree; it can be easily generalized to any degree. Note that the actions of $W'$ and $A_\lambda$ commute since the actions of $W'$ and $A_{\lambda^{h,i}}$ or $A_{\lambda^{v,i}}$ on the cohomology of Springer fibers corresponding to $G'$ commute by \cite[6.1]{spr}. (Note that $\tau \in A_{\lambda^{h,i}}$, thus $H^{k}(\lambda^{h,i})^\tau$ are still $W'$-modules. Also in Case I.$(a_2)$, if $\lambda^{v,i}$ is very even, then $A_{\lambda} = \{*\}$. Thus the statement is still clear in this case.)
\end{proof}

Now we consider the long exact sequences of the cohomology with compact support of $\B_N$ corresponding to the stratification by $X$ and $X_j$. By Proposition \ref{prop:Wiso}, these long exact sequences are defined in the category of $W'\times A_N$-modules. Thus it follows that $H^*(\B_N)$ is isomorphic to the sum of such $H^*_c(X)$ and $H^*_c(X_j)$ as a $W'\times A_N$-module. Furthermore, the second part of Lemma \ref{lem:hs} implies that the $W'$-module structure on $H^*(\B_N)$, as a sum of such $H^*_c(X)$ and $H^*_c(X_j)$, coincides with the restriction to $W'$ of the $W$-module structure defined by Springer theory on $H^*(\B_N)$. In sum, we have the following theorem.
\begin{thm} \label{thm:mainBCD} Let $N=N_\lambda \in \g$ where $\lambda=(1^{r_1}2^{r_2}\cdots)$. Define $\sgn_i$ to be the character of $W'\times A_\lambda$ such that on $A_\lambda$ it is the restriction of the character of $\tA_\lambda$ defined by
$$\sgn_i(z_i) = -1, \qquad \sgn_i(z_j) = 1 \textup{ for } i\neq j,$$
and on $W'$ it is trivial. Also let $\tau_i \colonequals z_i'z_{i-2}' \in A_{\lambda^{h,i}}$ and $\spr(\lambda^{h,i})^{\tau_i}$ be the character of $H^*(\lambda^{h,i})^{\tau_i}$.
Then we have the following equalities of characters of $W'\times A_\lambda$. Here we define
\begin{align*}
&\spr(\lambda^{h,i})(-,z_\alpha) \colonequals \spr(\lambda^{h,i})(-,z_\alpha'), &&\spr(\lambda^{v,i})(-,z_\alpha) \colonequals \spr(\lambda^{v,i})(-,z_\alpha') & \textup{ for } \alpha \neq i,
\\&\spr(\lambda^{h,i})(-,z_i) \colonequals \spr(\lambda^{h,i})(-,z_{i-2}'), &&\spr(\lambda^{v,i})(-,z_i) \colonequals \spr(\lambda^{v,i})(-,z_{i}').
\end{align*}

\begin{enumerate}[(a)]
\item Let $G=SO_{2n+1}$. Then,
\begin{align*}
\Res_{W' \times A_\lambda}^{W\times A_\lambda}\spr(\lambda) =& \sum_{i\geq 2,\ i \textup{ odd},\ r_i \textup{ odd}} \bigg( \spr(\lambda^{h,i}) +(r_i-1)\spr(\lambda^{v,i})\bigg)
\\&+ \sum_{i \geq 2,\ i \textup{ odd}, \ r_i \textup{ even}} \bigg( (1-\sgn_i)\spr(\lambda^{h,i})^{\tau_i} +(r_i-1+\sgn_i)\spr(\lambda^{v,i}) \bigg)
\\&+\sum_{i \textup{ even}} r_i\spr(\lambda^{v,i}).
\end{align*}

\item Let $G=Sp_{2n}$. Then,
\begin{align*}
\Res_{W' \times A_\lambda}^{W\times A_\lambda}\spr(\lambda) =& \sum_{i \textup{ even},\ r_i \textup{ odd}} \bigg( \spr(\lambda^{h,i}) +(r_i-1)\spr(\lambda^{v,i})\bigg)
\\&+ \sum_{i \textup{ even}, \ r_i \textup{ even}} \bigg( (1-\sgn_i)\spr(\lambda^{h,i})^{\tau_i} +(r_i-1+\sgn_i)\spr(\lambda^{v,i}) \bigg)
\\&+\sum_{i \textup{ odd}} r_i\spr(\lambda^{v,i}).
\end{align*}

\item Let $G=SO_{2n}$ and $\lambda$ is not very even. Then there is at most one $e\in \Z_{>0}$ such that $\lambda^{v,e}$ is very even. If such $e$ exists, then $e$ is odd, $r_e=2$, and $A_\lambda = \{*\}$. In this case we have
\begin{align*}
\Res_{W' \times A_\lambda}^{W\times A_\lambda}\spr(\lambda) =&\sum_{i \textup{ even}} r_i\spr(\lambda^{v,i}) + \spr(\lambda^{v,e}+)+\spr(\lambda^{v,e}-).
\end{align*}
If such $e$ does not exist, then we have
\begin{align*}
\Res_{W' \times A_\lambda}^{W\times A_\lambda}\spr(\lambda) =& \sum_{i\geq 2,\ i \textup{ odd},\ r_i \textup{ odd}}\bigg( \spr(\lambda^{h,i}) +(r_i-1)\spr(\lambda^{v,i}) \bigg)
\\&+ \sum_{i \geq 2,\ i \textup{ odd}, \ r_i \textup{ even}} \bigg(  (1-\sgn_i)\spr(\lambda^{h,i})^{\tau_i}+(r_i-1+\sgn_i) \spr(\lambda^{v,i}) \bigg)
\\&+\sum_{i \textup{ even}} r_i\spr(\lambda^{v,i}).
\end{align*}
%

If $\lambda$ is very even, then $A_\lambda=\{*\}$ and we have
\begin{align*}
\Res_{W' \times A_\lambda}^{W\times A_\lambda}\spr(\lambda+) =\Res_{W' \times A_\lambda}^{W\times A_\lambda}\spr(\lambda-) =\sum_{i \textup{ even}} r_i\spr(\lambda^{v,i}).
\end{align*}
\end{enumerate}
\end{thm}
Now Theorem \ref{thm:main} is a corollary of Theorem \ref{thm:mainBCD} if we evaluate equations above at $(id, id) \in W' \times A_\lambda$.

\begin{example} Let $G=Sp_{12}$ and $\lambda=(6,4,2) \vdash 12$. Then we have
\begin{align*}
\ec(\lambda) &=\ec(6,4)+\ec(6,2,2)+ \ec(4,4,2)
\\&=\ec(6,2)+2\ec(6,1,1)+2\ec(4,4)+\ec(4,2,2)+ 2\ec(3,3,2)
\\&=5\ec(6)+\ec(4,2)+4\ec(4,1,1)+6\ec(3,3)+5\ec(2,2,2)
\\&=14\ec(4)+18\ec(2,2)+14\ec(2,1,1)
\\&=42\ec(2)+50\ec(1,1)=142.
\end{align*}
\end{example}

\section{Betti numbers in some special cases}
In some special situation, we have not only the restriction of total Springer representations, i.e. the alternating sum $H^*(\B_N)$ of cohomology, but also that of each degree of the cohomology. If $G=GL_n$, it is already given in Theorem \ref{thm:bettiA}. Thus from now on we assume $G=SO_{2n+1}, Sp_{2n}, $ or $SO_{2n}$ and find analogous formulas.

We assume each of the following cases. Let $\lambda=(1^{r_1}2^{r_2} \cdots)$ and $\xi\geq 1$ be the smallest integer such that $r_\xi \neq 0$.
\begin{enumerate}[(a)]
\item $G=SO_{2n+1}$ or $SO_{2n}$.
\begin{enumerate}[($\textup{a}_1$)]
\item $\xi$ is even and $r_i\in\{0,1\}$ for any odd $i$. 
\item $\xi >1$ is odd, $r_i\in\{0,1\}$ for odd $i$ different from $\xi$, and $r_\xi\in\{1,2,3\}$.
\item $\xi=1$, $r_1=1$, and $\lambda^1$ (see \ref{note:partition}) satisfies either $(\textup{a}_1)$ or $(\textup{a}_2)$.
\item $\xi=1$, $r_1>1$, and $r_i\in\{0,1\}$ for odd $i$ different from 1.
\end{enumerate}
%
\item $G=Sp_{2n}$.
\begin{enumerate}[($\textup{b}_1$)]
\item $\xi$ is odd and $r_i\in\{0,1\}$ for any even $i$. 
\item $\xi$ is even, $r_i\in\{0,1\}$ for even $i$ different from $\xi$, and $r_\xi \in \{1,2,3\}$.
\end{enumerate}
\end{enumerate}

Recall that in the proof of Theorem \ref{thm:mainBCD} we used the long exact sequences of $W'\times A_N$-modules  to show that $H^*(\B_N)$ is the sum of alternating sums of the cohomology of each stratum in $\B_N$. It is easy to show that in each case above, such a long exact sequence splits into short exact sequences since Springer fibers have vanishing odd cohomology \cite{c-l-p}. Thus in this case we have the following theorem.

\begin{thm} \label{thm:bettiBCD} Suppose $G=SO_{2n+1}, Sp_{2n},$ or $SO_{2n}$. Let $\lambda=(1^{r_1}2^{r_2}\cdots)$ be a partition and $N=N_\lambda \in \g$. We define $d_i \colonequals \sum_{j>i} r_j$. Also recall the definition of $\sgn_i$ and $\tau_i  \in A_{\lambda^{h,i}}$ in Theorem \ref{thm:mainBCD}. Here we abuse notations to denote by $H^k(\lambda), H^k(\lambda^{h,i}), H^k(\lambda^{v,i})$ the character corresponding to each representation and define the character values of elements in $A_\lambda$ on $H^{k}(\lambda^{h,i}), H^{k}(\lambda^{v,i})$ as follows.
\begin{align*}
&H^{k}(\lambda^{h,i})(-,z_\alpha) \colonequals H^{k}(\lambda^{h,i})(-,z_\alpha'), &&H^{k}(\lambda^{v,i})(-,z_\alpha) \colonequals H^{k}(\lambda^{v,i})(-,z_\alpha') & \textup{ for } \alpha \neq i,
\\&H^{k}(\lambda^{h,i})(-,z_i) \colonequals H^{k}(\lambda^{h,i})(-,z_{i-2}'), &&H^{k}(\lambda^{v,i})(-,z_i) \colonequals H^{k}(\lambda^{v,i})(-,z_{i}').
\end{align*}
Then we have equalities of characters of $W'\times A_\lambda$ as follows.
\begin{enumerate}[$(\textup{a})$] 
\item Assume $G=SO_{2n+1}$ or $SO_{2n}$.
\begin{enumerate}[$(\textup{a}_1)$] 
\item If $\xi$ is even and $r_i\in\{0,1\}$ for any odd $i$, then 
$$\Res_{W' \times A_\lambda}^{W\times A_\lambda}H^k(\lambda)= \sum_{i >1 \textup{ odd},\ r_i=1} H^{k-2d_i}(\lambda^{h, i}) + \sum_{i \textup{ even}} \sum_{j=0}^{r_i-1} H^{k-2d_i-2j}(\lambda^{v,i}).$$


\item Suppose $\xi$ is odd, $\xi >1$, $r_i\in\{0,1\}$ for odd $i$ different from $\xi$, and $r_\xi\in\{1,2,3\}$. If $r_\xi=1$, then the formula in $(\textup{a}_1)$ is still valid. If $r_\xi=2$, then
\begin{align*}
\Res_{W' \times A_\lambda}^{W\times A_\lambda} H^k(\lambda) =&\sum_{i > \xi \textup{ odd},\ r_i=1} H^{k-2d_i}(\lambda^{h, i}) + \sum_{i  \textup{ even}} \sum_{j=0}^{r_i-1} H^{k-2d_i-2j}(\lambda^{v,i}) 
\\&+ H^{k-2d_\xi-2}(\lambda^{h,\xi})^{\tau_\xi}-\sgn_\xi H^{k-2d_\xi}(\lambda^{h,\xi})^{\tau_\xi}+ (1+\sgn_\xi)H^{k-2d_\xi}(\lambda^{v, \xi}) .
\end{align*}
If $r_\xi=3$, then 
\begin{align*}
\Res_{W' \times A_\lambda}^{W\times A_\lambda}H^k(\lambda)=&\sum_{i > \xi \textup{ odd},\ r_i=1} H^{k-2d_i}(\lambda^{h, i}) + \sum_{i \textup{ even}} \sum_{j=0}^{r_i-1} H^{k-2d_i-2j}(\lambda^{v,i}) 
\\&+ H^{k-2d_\xi-4}(\lambda^{h,\xi})^{\tau_\xi} -H^{k-2d_\xi-2}(\lambda^{h,\xi})^{\tau_\xi}+H^{k-2d_\xi-2}(\lambda^{h,\xi})
\\&+ H^{k-2d_\xi}(\lambda^{v, \xi}) +H^{k-2d_\xi-2}(\lambda^{v, \xi}) .
\end{align*}

\item Suppose $\xi=1$, $r_1=1$, and $\lambda^1$ satisfies $(\textup{a}_1)$ (resp. $(\textup{a}_2)$). Then the formula in $(\textup{a}_1)$ (resp. $(\textup{a}_2)$) is still valid if we replace $\xi$ by the smallest integer $\xi'$ such that $r_{\xi'}\neq0$ and $\xi'>1$.

\item Suppose $\xi=1$, $r_1>1$, and  $r_i\in\{0,1\}$ for odd $i$ different from 1. If $r_1$ is odd, then
$$\Res_{W' \times A_\lambda}^{W\times A_\lambda}H^k(\lambda) = \sum_{i>1 \textup{ odd},\ r_i =1} H^{k-2d_i}(\lambda^{h, i}) + \sum_{i \textup{ even}} \sum_{j=0}^{r_i-1} H^{k-2d_i-2j}(\lambda^{v,i}) + \sum_{j=0}^{r_1-2}H^{k-2d_1-2j}(\lambda^{v, 1}).$$
If $r_1$ is even, then
\begin{align*}
\Res_{W' \times A_\lambda}^{W\times A_\lambda}H^k(\lambda) =&\sum_{i>1 \textup{ odd},\ r_i=1} H^{k-2d_i}(\lambda^{h, i}) + \sum_{i \textup{ even}} \sum_{j=0}^{r_i-1} H^{k-2d_i-2j}(\lambda^{v,i}) 
\\&+ \sum_{j=0}^{r_1-2}H^{k-2d_1-2j}(\lambda^{v, 1}) + \sgn_1H^{k-2d_1-r_1+2}(\lambda^{v, 1}).
\end{align*}
\end{enumerate}
\item Assume $G=Sp_{2n}$.
\begin{enumerate}[$(\textup{b}_1)$]
\item If $\xi$ is odd and $r_i\in\{0,1\}$ for any even $i$, then
$$\Res_{W' \times A_\lambda}^{W\times A_\lambda}H^k(\lambda)= \sum_{i \textup{ even},\ r_i =1} H^{k-2d_i}(\lambda^{h, i}) + \sum_{i \textup{ odd}} \sum_{j=0}^{r_i-1} H^{k-2d_i-2j}(\lambda^{v,i}).$$

\item Suppose $\xi$ is even, $r_i\in\{0,1\}$ for even $i$ different from $\xi$, and $r_\xi \in \{1,2,3\}$. If $r_\xi=1$, then the formula in $(\textup{b}_1)$ is still valid. If $r_\xi=2$, then
\begin{align*}
\Res_{W' \times A_\lambda}^{W\times A_\lambda}H^k(\lambda) =&\sum_{i > \xi \textup{ even},\ r_i=1} H^{k-2d_i}(\lambda^{h, i}) + \sum_{i \textup{ odd}} \sum_{j=0}^{r_i-1} H^{k-2d_i-2j}(\lambda^{v,i}) 
\\&+ H^{k-2d_\xi-2}(\lambda^{h,\xi})^{\tau_\xi}-\sgn_\xi H^{k-2d_\xi}(\lambda^{h,\xi})^{\tau_\xi}+ (1+\sgn_\xi)H^{k-2d_\xi}(\lambda^{v, \xi}) .
\end{align*}
If $r_\xi=3$, then 
\begin{align*}
\Res_{W' \times A_\lambda}^{W\times A_\lambda}H^k(\lambda) =&\sum_{i > \xi \textup{ even},\ r_i=1} H^{k-2d_i}(\lambda^{h, i}) + \sum_{i \textup{ odd}} \sum_{j=0}^{r_i-1} H^{k-2d_i-2j}(\lambda^{v,i}) 
\\&+ H^{k-2d_\xi-4}(\lambda^{h,\xi})^{\tau_\xi} -H^{k-2d_\xi-2}(\lambda^{h,\xi})^{\tau_\xi}+H^{k-2d_\xi-2}(\lambda^{h,\xi})
\\&+ H^{k-2d_\xi}(\lambda^{v, \xi}) +H^{k-2d_\xi-2}(\lambda^{v, \xi}) .
\end{align*}
\end{enumerate}
\end{enumerate}
\end{thm}
In Section \ref{sec:closed} we use these formulas to calculate the Betti numbers of Springer fibers corresponding to two-row partitions.

%
%


%

\section{Short proof of the main theorem}
Indeed, the proof is simpler if we only want to prove Theorem \ref{thm:main}. First we recall the following induction statement of Springer representations from \cite[Theorem 1.3]{lu:indthm}.
\begin{prop}\label{prop:ind} Let $L$ be a Levi subgroup of a parabolic subgroup of $G$ with its Lie algebra $\mathfrak{l}$. Let $W_L$ be the Weyl group of $L$ with a natural embedding $W_L \hookrightarrow W$. Let $\B_L$ be the variety of Borel subgroups of $L$. Then for $N \in \mathfrak{l} \subset \g$ we have
$$H^*(\B_N) \simeq \Ind_{W_L}^W H^*((\B_L)_N)$$
as $W_L$-modules.
\end{prop}
Suppose $G=GL_n$. As every nilpotent element in $\g$ is a regular element in a Levi subalgebra of some parabolic subalgebra of $\g$, for $\lambda=(1^{r_1}2^{r_2} \cdots)$ it is easy to show that
$$\ec(\lambda) = \frac{n!}{(1!)^{r_1}(2!)^{r_2} \cdots}.$$
Thus Theorem \ref{thm:main} follows from easy induction on $n$.

Now if $G=SO_{2n+1}, Sp_{2n}, $ or $SO_{2n}$, using Proposition \ref{prop:ind} it suffices to show the statement when the given nilpotent element is distinguished. In this case we follow argument in Section \ref{sec:BCD}, which is a bit simpler than considering all the cases.

\begin{rmk} Note that vanishing of odd cohomology of Springer fibers \cite{c-l-p} implies the positivity of Euler characteristics of Springer fibers. Meanwhile, we do not need this fact in the proof of Theorem \ref{thm:mainA} and \ref{thm:mainBCD}. Thus it gives another proof that the Euler characteristic of any Springer fiber is positive without using \cite{c-l-p} (at least if $G$ is of classical type).
\end{rmk}

\section{Closed formula for the Betti numbers in two-row cases} \label{sec:closed}
Let $N=N_\lambda \in \g$ be a nilpotent element corresponding to $\lambda$, which consists of two rows (or if $G=SO_{2n+1}$, we assume that it consists of two rows with additional 1). Here we use Theorem \ref{thm:bettiA} and Theorem \ref{thm:bettiBCD} to give closed formulas for the multiplicities of irreducible representations of $A_\lambda$ in each $H^{k}(\lambda)$. As a result, we also have formulas for Betti numbers of Springer fibers of such type. There are many results about geometry of such Springer fibers, e.g. \cite{fung}, \cite{khovanov}, \cite{fresse2}, \cite{st-web}, \cite{russell}, \cite{eh-st}, \cite{wilbert}, \cite{wilbert2}, etc.

\begin{prop} Let $G=GL_n$ and $\lambda = (i, j) \vdash n$ such that $i\geq j\geq 0$. Then 
$$h^{2k}(\lambda)= {i+j \choose k}-{i+j \choose k-1}$$
for $0\leq k \leq j$ and $0$ otherwise.
\end{prop}
\begin{rmk} This is also calculated in \cite[Example 4.5]{fresse}.
\end{rmk}
\begin{proof} We use induction on the rank $n = i+j$. It is trivial when $n=1$. Now suppose $n \geq 2$ and that the statement is true up to rank $n-1$. If $\lambda=(i,j)\vdash n$ with $i>j$, then by Theorem \ref{thm:bettiA} we have
$$h^{2k}(\lambda) = h^{2k}(i-1,j)+h^{2k-2}(i,j-1),$$
thus $h^{2k}(\lambda)=0$ unless $0 \leq k \leq j$. If $k$ is in this range then
$$h^{2k}(\lambda) = {i+j-1\choose k}-{i+j-1 \choose k-2}={i+j \choose k} - {i+j \choose k-1}.$$
Likewise, if $\lambda = (i,i)$ then 
$$h^{2k}(\lambda) = h^{2k}(i,i-1)+h^{2k-2}(i,i-1) = {2i \choose k} - {2i \choose k-1}$$
if $0 \leq k \leq i$. (Note that this is true even when $k=i$ as ${2i-1 \choose i-1}-{2i-1 \choose i-2}= {2i \choose i} - {2i \choose i-1}$.)
\end{proof}

Let $G=SO_{2n+1}$ and $\lambda=(i,j,1) \vdash 2n+1$ where $i\geq j\geq 1$. Then $A_\lambda$ is isomorphic to 
\begin{align*}
&\Z/2\times \Z/2 &&\textup{ if }i>j>1,
\\ &\Z/2 &&\textup{ if } i \textup{ odd and either } i=j>1 \textup{ or } i>j=1, \textup{ and}
\\ &\textup{trivial } &&\textup{ otherwise.}
\end{align*}
If $i>j>1$, then we let $z_iz_j$ (resp. $z_jz_1$) be the generator of the first (resp. second) factor of $\Z/2$. Then we set 
\begin{align*}
h^k(\lambda)_{+,+} &\colonequals \br{H^k(\lambda), Id\times Id}_{A_\lambda}, &h^k(\lambda)_{+,-}&\colonequals\br{H^k(\lambda),Id\times \sgn}_{A_\lambda},
\\h^k(\lambda)_{-,+}&\colonequals\br{H^k(\lambda),\sgn\times Id}_{A_\lambda},&h^k(\lambda)_{-,-}&\colonequals \br{H^k(\lambda),\sgn\times \sgn}_{A_\lambda},
\end{align*}
to be the multiplicity of $Id \times Id, Id\times \sgn, \sgn \times Id, \sgn\times \sgn$, respectively, in $H^k(\lambda)$ as an $A_\lambda$-representation. If $i=j>1$ is odd, then we write 
\begin{align*}
h^k(\lambda)_{+,+} \colonequals\br{H^k(\lambda), Id}_{A_\lambda},&& h^k(\lambda)_{+,-} \colonequals\br{H^k(\lambda), \sgn}_{A_\lambda}
\end{align*}
to be the multiplicity of $Id, \sgn$, respectively, in $H^k(\lambda)$, and set $h^k(\lambda)_{-,+}= h^k(\lambda)_{-,-}=0$. If $i$ is odd and $j=1$, then we let 
\begin{align*}
H^k(\lambda)_{+,+}\colonequals\br{H^k(\lambda) Id}_{A_\lambda},&& h^k(\lambda)_{-,+}\colonequals\br{H^k(\lambda),\sgn}_{A_\lambda}
\end{align*} be the multiplicity of $Id, \sgn$, respectively, in $H^k(\lambda)$, and set $h^k(\lambda)_{+,-}= h^k(\lambda)_{-,-}=0$. If $A_\lambda$ is trivial, we set $h^k(\lambda)_{+,+}=h^k(\lambda)$ and $h^k(\lambda)_{+,-}=h^k(\lambda)_{-,+}=h^k(\lambda)_{-,-}=0$. (The reason for using these weird notations will be apparent immediately.) Then we have the following.

\begin{prop} Let $G=SO_{2n+1}$ and $\lambda=(i,j,1) \vdash 2n+1$ such that $i\geq j \geq 1$, and $i,j$ are odd unless $i=j$. Then $h^{\alpha}(\lambda)=0$ unless $0 \leq \alpha \leq j+2$ and $\alpha$ even. From now on we assume $\alpha =2k$ satisfies this condition.
\begin{enumerate}[(a)]
\item If $i>j>1$, then we have
\begin{align*}
h^{2k}(\lambda)_{+,+}={\frac{i+j}{2} \choose k}, \qquad h^{2k}(\lambda)_{+,-} = {\frac{i+j}{2} \choose k-2}, \qquad h^{2k}(\lambda)_{-,-}=0.
\end{align*}
Also, $h^{2k}(\lambda)_{-,+}=0$ unless $2k=j+1$, in which case 
$$h^{j+1}(\lambda)_{-,+}=\frac{i-j}{i+j+2}{\frac{i+j+2}{2} \choose \frac{j+1}{2}}.$$
\item If $i=j$ is odd and $i>1$, then 
$$h^{2k}(\lambda)_{+,+}= {i \choose k}, \qquad h^{2k}(\lambda)_{+,-} = {i \choose k-2}.$$
\item If $i>1$ is odd and $j=1$, then $$h^{2k}(\lambda)_{+,+} = {\frac{i+1}{2} \choose k}.$$ Also $h^{2k}(\lambda)_{-,+}=0$ unless $2k=2$, in which case $h^{2}(\lambda)_{-,+} = \frac{i-1}{2}$.
\item If $i=j>1$ is even, then $h^{2k}(\lambda)={i \choose k}.$
\item $h^{0}((1,1,1))=h^{2}((1,1,1))=1$.
\end{enumerate}
Thus, we always have $h^{2k}(\lambda)_{-,-}=0$ and
$$ h^{2k}(\lambda)_{+,+}={\frac{i+j}{2} \choose k}, \qquad h^{2k}(\lambda)_{+,-} = {\frac{i+j}{2} \choose k-2}, \qquad h^{2k}(\lambda)_{-,+}=\delta_{2k, j+1} \frac{i-j}{i+j+2}{\frac{i+j+2}{2} \choose \frac{j+1}{2}}$$
if they are not a priori zero.
\end{prop}
\begin{proof} From Theorem \ref{thm:bettiBCD} we have the following relations.
\begin{enumerate}[(a)]
\item If $i>j>1$, then
$$h^{2k}(\lambda)_* = h^{2k}(i-2,j,1)_*+h^{2k-2}(i,j-2,1)_*$$
where $*$ can be any of $(+,+), (+,-), (-,+), (-,-)$.
\item If $i=j$ is odd and $i>1$, then 
\begin{align*}
h^{2k}(\lambda)_{+,+} &= h^{2k}(i-1,i-1,1)_{+,+}+h^{2k-2}(i,i-2,1)_{+,+}-h^{2k}(i,i-2,1)_{+,-},
\\h^{2k}(\lambda)_{+,-} &= h^{2k}(i-1,i-1,1)_{+,+}+h^{2k-2}(i,i-2,1)_{+,-}-h^{2k}(i,i-2,1)_{+,+}.
\end{align*}
\item If $i>1$ is odd and $j=1$, then
\begin{align*}
h^{2k}(\lambda)_* &= h^{2k}(i-2,1,1)_*+\delta_{k,1},
\end{align*}
where $*$ is either $(+,+)$ or $(-,+)$.
\item If $i=j$ is even, then
$$h^{2k}(\lambda) = h^{2k-2}(i-1,i-1,1)_{+,+}+ h^{2k}(i-1,i-1,1)_{+,+}.$$
\item If $i=j=1$, then $h^{0}((1,1,1))=h^{2}((1,1,1))=1$.
\end{enumerate}
Now the result follows from induction on the rank $n$.
\end{proof}

If $G=Sp_{2n}$, it is known that the $A_N$-action on $H^k(\B_N)$ factors through the quotient by the image of $\pm I \in Sp_{2n}$, where the image of $-I$ is $\prod_i z_i^{r_i} \in A_N$. We denote such a quotient by $\oA_N$. If $N=N_\lambda$ for $\lambda=(i,j) \vdash n$, then $\oA_\lambda \simeq \Z/2$ if $i,j>0$ are both even, and otherwise $\oA_{\lambda}$ is trivial. If $\oA_\lambda \simeq \Z/2$, then we let $h^k(\lambda)_{Id}, h^k(\lambda)_{\sgn}$ be the multiplicities of $Id, \sgn$, respectively, in $H^k(\lambda)$ as a $\oA_\lambda$-module. Thus in particular $h^k(\lambda) = h^k(\lambda)_{Id}+h^k(\lambda)_{\sgn}$. If $\oA_\lambda$ is trivial, we set $h^k(\lambda)_{Id} \colonequals h^k(\lambda)$ and $h^k(\lambda)_{\sgn} \colonequals 0$.

\begin{prop} Let $G=Sp_{2n}$ and $\lambda=(i,j) \vdash 2n$ such that $i\geq j \geq 0$, and $i, j$ are even unless $i=j$. Then we have $h^{2k}(\lambda)_{id}=0$ unless $0\leq 2k \leq j+1$, in which case
\begin{align*}
\textup{if } 0 \leq 2k\leq j, &\qquad h^{2k}(\lambda)_{Id}={ \floor{\frac{i+1}{2}}+\floor{\frac{j+1}{2}} \choose k},
\\\textup{if } 2k=j+1, &\qquad h^{2k}(\lambda)_{Id}=\frac{1}{2}{ \floor{\frac{i+1}{2}}+\floor{\frac{j+1}{2}} \choose k}.
\end{align*}
If $i, j$ are both even, then $h^{2k}_{\sgn}(\lambda)=0$ unless $0 \leq 2k \leq j$, in which case
$$h^{2k}(\lambda)_{\sgn} = {\frac{i+j}{2} \choose k-1}.$$
\end{prop}
\begin{proof} From Theorem \ref{thm:bettiBCD} we have the following relations.
\begin{enumerate}[(a)]
\item If $i >j$, then
\begin{align*}
h^{2k}(\lambda)_{Id} &= h^{2k}(i-2,j)_{Id} +h^{2k-2}(i,j-2)_{Id},
\\h^{2k}(\lambda)_{\sgn} &= h^{2k}(i-2,j)_{\sgn} +h^{2k-2}(i,j-2)_{\sgn}. 
\end{align*}
\item If $i=j$ is even, then
\begin{align*}
h^{2k}(\lambda)_{Id} &= h^{2k-2}(i,i-2)_{Id}- h^{2k}(i,i-2)_{\sgn} +h^{2k}(i-1,j-1),
\\h^{2k}(\lambda)_{\sgn} &= h^{2k-2}(i,i-2)_{\sgn}- h^{2k}(i,i-2)_{Id} +h^{2k}(i-1,j-1),
\end{align*}
\item If $i=j$ is odd, then
$$h^{2k}(\lambda) = h^{2k-2}(i-1,i-1)_{Id}+h^{2k-2}(i-1,i-1)_{\sgn}+h^{2k}(i-1,i-1)_{Id}+h^{2k}(i-1,i-1)_{\sgn}.$$
\end{enumerate}
Now the result follows from easy induction on $n$.
\end{proof}

If $G=SO_{2n}$, it is known that any the $A_N$-action on $H^k(\B_N)$ factors through the quotient by the image of $\pm I \in SO_{2n}$, where the image of $-I$ is $\prod_i z_i^{r_i}\in A_N$. We denote such a quotient by $\oA_N$. If $N=N_\lambda$ for $\lambda=(i,j) \vdash n$, then $\oA_\lambda$ is trivial.
\begin{prop} Let $G=SO_{2n}$ and $\lambda=(i,j) \vdash 2n$ such that $i\geq j \geq 0$, and $i, j$ are odd unless $i=j$. Then we have $h^{2k}(\lambda)=0$ unless $0\leq 2k \leq j$, in which case
\begin{align*}
\textup{if } 0 \leq 2k\leq j-1, &\qquad h^{2k}(\lambda)={ \frac{i+j}{2} \choose k},
\\\textup{if } 2k=j, &\qquad h^{2k}(\lambda)=\frac{1}{2}{ \frac{i+j}{2} \choose k}.
\end{align*}
\end{prop}
\begin{proof} From Theorem \ref{thm:bettiBCD} we have the following relations.
\begin{enumerate}[(a)]
\item If $i>j$, then
$$h^{2k}(\lambda) = h^{2k}(i-2,j)+h^{2k-2}(i,j-2).$$
\item If $i=j$ is odd, then
$$h^{2k}(\lambda) = h^{2k-2}(i,j-2)-h^{2k}(i,j-2)+2h^{2k}(i-1,j-1).$$
\item If $i=j$ is even, then
$$h^{2k}(\lambda) = h^{2k-2}(i-1,j-1)+h^{2k}(i-1,j-1).$$
\end{enumerate}
Now the result follows from easy induction on $n$.
\end{proof}

\begin{rmk} When $G=Sp_{2n}$ or $SO_{2n}$, the cohomology rings of Springer fibers corresponding to two-row partitions are described in \cite[Theorem B]{eh-st} and \cite[Theorem A]{wilbert2}. Thus their Betti numbers can also be deduced from their descriptions.
\end{rmk}

%


%
%
%


\appendix
\section{Combinatorial proof of Theorem \ref{thm:mainA}} \label{app:comA}
\begin{proof}[Proof of Theorem \ref{thm:mainA}]We refer to \cite[Chapter 7.18]{stanley}. We let $\S_i$ be the symmetric group of $i$ elements. Define $CF_i$ to be the $\Q$-vector space spanned by irreducible characters of $\S_i$ and let $CF\colonequals \bigoplus_{i\in\N} CF_i$. We introduce a ring structure on $CF$ so that for $f \in CF_i$ and $g \in CF_j$, we define $f \circ g \colonequals \Ind^{\S_{i+j}}_{\S_{i}\times \S_j} f\times g$ and extend it to $CF$ by linearity. Also on each $CF_i$ we have a well-defined inner product $\br{\ , \ }$, which we extend to $CF$ be decreeing that $\br{f, g}=0$ for $f\in CF_i, g\in CF_j$ with $i\neq j$.

Then we have a well-defined ring isomorphism
$$\textup{ch}: CF \rightarrow \Lambda$$
where $\Lambda$ is the ring of symmetric polynomials of $x_1, x_2, \cdots$ with coefficients in $\Q$. It is uniquely defined by the condition that it sends $\Ind_{\S_\lambda}^{\S_n} Id_{\S_\lambda}$ to $h_\lambda$, the homogeneous symmetric polynomial corresponding to $\lambda$. Here $\S_\lambda\colonequals \S_{1}^{r_1} \times \S_{2}^{r_2}\times \cdots \subset \S_n$ if $\lambda=(1^{r_1}2^{r_2}\cdots) \vdash n$. Also $\textup{ch}$ respects inner products on both rings, where the inner product on $\Lambda$ is defined by $\br{h_{\lambda}, m_{\mu}}=\delta_{\lambda, \mu}$ and extended by linearity. Here $m_\mu$ is the monomial symmetric polynomial corresponding to $\mu$.

In order to prove the statement, it suffices to show that for any class function $f$ of $\S_{n-1}$
$$\br{ \spr(\lambda),\Ind_{\S_{n-1}}^{\S_n}f} = \br{ \bigoplus_{r_i \geq 1} r_i \spr(\lambda^i), f}.$$
Note that $\spr(\lambda) = \Ind_{\S_{\lambda}}^{\S_n} Id_{\S_\lambda}$ by Proposition \ref{prop:ind}. If we apply $\textup{ch}$, then it is equivalent to
$$\br{h_\lambda,\textup{ch}(f)h_1} = \br{ \sum_{r_i \geq 1} r_i h_{\lambda^i}, \textup{ch}(f)}.$$
As monomial symmetric functions $m_\mu$ for $\mu \vdash n-1$ are a basis of $CF_{n-1}$, it suffices to check that the above formula is true when $\textup{ch}(f)=m_\mu$ for any $\mu \vdash n-1$. By the definition of inner product on $\Lambda$ we have $\br{ \sum_{r_i \geq 1} r_i h_{\lambda^i}, m_\mu}= r_i$ if $\mu=\lambda^i$ for some $i\in \N$ and $0$ otherwise. On the other hand, for $\mu = (1^{r_1'}2^{r_2'}\cdots)$ it is easy to show that
$$m_\mu h_1 = \sum_{r_i'\geq 1} (r_{i+1}'+1)m_{(1^{r_1'}2^{r_2'} \cdots i^{r'_i-1} (i+1)^{r'_{i+1}+1} \cdots)} +(r_{1}'+1)m_{(1^{r_1'+1}2^{r_2'} \cdots )},$$
thus $\br{h_\lambda,m_\mu h_1}=r_i$ if $\mu = \lambda^i$ for some $i \in \N$ and $0$ otherwise. It suffices for the proof.
\end{proof}

\section{Exceptional types} \label{sec:exceptional}
In this section we assume that $G$ is a reductive group of exceptional type, i.e. of $E_6, E_7, E_8, F_4$, or $G_2$, over $\k$ such that $\ch \k$ is good. Then the corresponding Green functions are completely known. The following tables give the multiplicities of each irreducible character of $A_N$ in $H^*(\B_N)$ for any nilpotent $N \in \g$ if $G$ is of exceptional type, using the data of Green functions. Here we use the tables of Green functions given in \cite{lubeck}. Each column in the tables means the following.
\begin{enumerate}[(a)]
\item $N$: the type of a nilpotent element $N \in \g$. We use the Bala-Carter notation here.
\item $A_N$: the component group of the stabilizer of $N$ in $G$. A dot (.) means that $A_N$ is trivial. Otherwise $A_N$ is either $\S_2, \S_3, \S_4$, or $\S_5$, where $\S_n$ is the symmetric group of $n$ elements.
\item $\phi \in \widehat{A_N}$: an irreducible character of $A_N$. If $A_N$ is trivial, then $\phi$ is the identity, and we put a dot (.) in this case. Otherwise if $A_N=\S_n$, then we put the partition $\lambda \vdash n$ which parametrizes $\phi$. For example, $n$ means the identity character and $11 \cdots 1$ means the sign character.
\item $\br{H^*(\B_N),\phi}$: the multiplicity of $\phi$ in $H^*(\B_N)$.
\end{enumerate}
Note that in each case the Euler characteristic of $\B_N$ is given by $\sum_{\phi \in \widehat{A_N}} (\dim \phi)\br{H^*(\B_N),\phi}_{A_N}$ where the sum is over all the irreducible representations of $A_N$.

\newpage
\begin{longtable}{|c|c|c|c|}
\caption{Type $E_6$} \label{table:e6} \\
\hline $N$ & $A_N$ & $\phi \in \widehat{A_N}$ &$\br{H^*(\B_N),\phi}$ \\ \hline \hline
$E_6$&.&.&1\\\hline
$E_6(a_1)$&.&.&7\\\hline
$D_5$&.&.&27\\\hline
\multirow{2}{*}{$A_5+A_1$}&\multirow{2}{*}{$\S_2$} &2&57\\
&&11&21\\\hline
$A_5$&.&.&72\\\hline
$D_5(a_1)$&.&.&162\\\hline
$A_4+A_1$&.&.&216\\\hline
$D_4$&.&.&270\\\hline
$A_4$&.&.&432\\\hline
\multirow{3}{*}{$D_4(a_1)$}&\multirow{3}{*}{$\S_3$} &3&575\\
&&21&370\\
&&111&35\\\hline
$A_3+A_1$&.&.&1080\\\hline
$2A_2+A_1$&.&.&720\\\hline
$A_3$&.&.&2160\\\hline
$A_2+2A_1$&.&.&2160\\\hline
$2A_2$&.&.&1440\\\hline
$A_2+A_1$&.&.&4320\\\hline
\multirow{2}{*}{$A_2$}&\multirow{2}{*}{$\S_2$}&2&5940\\
&&11&2700\\\hline
$3A_1$&.&.&6480\\\hline
$2A_1$&.&.&12960\\\hline
$A_1$&.&.&25920\\\hline
$A_0$&.&.&51840\\\hline
\end{longtable}

\newpage
\newsavebox\ltmcbox

\begin{multicols}{2}
\medskip
\setbox\ltmcbox\vbox{
\makeatletter\col@number\@ne
\begin{longtable}{|c|c|c|c|}
\caption{Type $E_7$} \label{table:e7} \\
\hline $N$ & $A_N$ & $\phi \in \widehat{A_N}$ &$\br{H^*(\B_N),\phi}$ \\ \hline \hline
$E_7$&.&.&1\\\hline
$E_7(a_1)$&.&.&8\\\hline
$E_7(a_2)$&.&.&35\\\hline
\multirow{2}{*}{$D_6+A_1$}&\multirow{2}{*}{$\S_2$}&2&91\\
&&11&28\\\hline
$E_6$&.&.&56\\\hline
\multirow{2}{*}{$E_6(a_1)$}&\multirow{2}{*}{$\S_2$}&2&232\\
&&11&160\\\hline
$D_6$&.&.&126\\\hline
\multirow{2}{*}{$D_6(a_1)+A_1$}&\multirow{2}{*}{$\S_2$}&2&456\\
&&11&15\\\hline
$A_6$&.&.&576\\\hline
$D_6(a_1)$&.&.&882\\\hline
$D_5+A_1$&.&.&756\\\hline
\multirow{3}{*}{$D_6(a_2)+A_1$}&\multirow{3}{*}{$\S_3$} &3&1442\\
&&21&826\\
&&111&56\\\hline
$D_5$&.&.&1512\\\hline
\multirow{2}{*}{$(A_5+A_1)'$}&\multirow{2}{*}{$\S_2$}&2&3192\\
&&11&1176\\\hline
$D_6(a_2)$&.&.&2772\\\hline
$(A_5+A_1)''$&.&.&2016\\\hline
$A_5'$&.&.&4032\\\hline
$D_5(a_1)+A_1$&.&.&4536\\\hline
\multirow{2}{*}{$D_5(a_1)$}&\multirow{2}{*}{$\S_2$}&2&6426\\
&&11&2646\\\hline
$A_4+A_2$&.&.&4032\\\hline
\multirow{2}{*}{$A_4+A_1$}&\multirow{2}{*}{$\S_2$}&2&7308\\
&&11&4788\\\hline
$A_5''$&.&.&4032\\\hline
$D_4+A_1$&.&.&7560\\\hline
\multirow{2}{*}{$A_4$}&\multirow{2}{*}{$\S_2$}&2&14616\\
&&11&9576\\\hline
$A_3+A_2+A_1$&.&.&10080\\\hline
\multirow{2}{*}{$A_3+A_2$}&\multirow{2}{*}{$\S_2$}&2&19530\\
&&11&630\\\hline
$D_4$&.&.&15120\\\hline
\multirow{2}{*}{$D_4(a_1)+A_1$}&\multirow{2}{*}{$\S_2$}&2&26460\\
&&11&11340\\\hline
$A_3+2A_1$&.&.&30240\\\hline
\multirow{3}{*}{$D_4(a_1)$}&\multirow{3}{*}{$\S_3$} &3&32200\\
&&21&20720\\
&&111&1960\\\hline
$(A_3+A_1)'$&.&.&60480\\\hline
$2A_2+A_1$&.&.&40320\\\hline
$(A_3+A_1)''$&.&.&60480\\\hline
$A_3$&.&.&120960\\\hline
$2A_2$&.&.&80640\\\hline
$A_2+3A_1$&.&.&60480\\\hline
$A_2+2A_1$&.&.&120960\\\hline
\multirow{2}{*}{$A_2+A_1$}&\multirow{2}{*}{$\S_2$}&2&166320\\
&&11&75600\\\hline
$4A_1$&.&.&181440\\\hline
\multirow{2}{*}{$A_2$}&\multirow{2}{*}{$\S_2$}&2&332640\\
&&11&151200\\\hline
$3A_1'$&.&.&362880\\\hline
$3A_1''$&.&.&362880\\\hline
$2A_1$&.&.&725760\\\hline
$A_1$&.&.&1451520\\\hline
$A_0$&.&.&2903040\\\hline
\end{longtable}
\unskip
\unpenalty
\unpenalty}
\unvbox\ltmcbox
\medskip
\end{multicols}

\newpage
\begin{multicols}{2}
\medskip
\setbox\ltmcbox\vbox{
\makeatletter\col@number\@ne
\begin{longtable}{|c|c|c|c|}
\caption{Type $E_8$} \label{table:e8} \\
\hline $N$ & $A_N$ & $\phi \in \widehat{A_N}$ &$\br{H^*(\B_N),\phi}$ \\ \hline \hline
\endfirsthead
%
$E_8$&.&.&1\\\hline
$E_8(a_1)$&.&.&9\\\hline
$E_8(a_2)$&.&.&44\\\hline
\multirow{2}{*}{$E_7+A_1$}&\multirow{2}{*}{$\S_2$}&2&156\\
&&11&36\\\hline
$E_7$&.&.&240\\\hline
\multirow{2}{*}{$D_8$}&\multirow{2}{*}{$\S_2$}&2&366\\
&&11&231\\\hline
\multirow{2}{*}{$E_7(a_1)+A_1$}&\multirow{2}{*}{$\S_2$}&2&1010\\
&&11&50\\\hline
$E_7(a_1)$&.&.&1920\\\hline
\multirow{2}{*}{$D_8(a_1)$}&\multirow{2}{*}{$\S_2$}&2&1710\\
&&11&495\\\hline
$D_7$&.&.&2160\\\hline
\multirow{3}{*}{$E_7(a_2)+A_1$}&\multirow{3}{*}{$\S_3$} &3&4284\\
&&21&2128\\
&&111&84\\\hline
\multirow{3}{*}{$A_8$}&\multirow{3}{*}{$\S_3$} &3&5589\\
&&21&4263\\
&&111&594\\\hline
$E_7(a_2)$&.&.&8400\\\hline
$E_6+A_1$&.&.&6720\\\hline
\multirow{2}{*}{$D_7(a_1)$}&\multirow{2}{*}{$\S_2$}&2&15780\\
&&11&1500\\\hline
\multirow{3}{*}{$D_8(a_3)$}&\multirow{3}{*}{$\S_3$} &3&14205\\
&&21&175\\
&&111&1650\\\hline
\multirow{2}{*}{$D_6+A_1$}&\multirow{2}{*}{$\S_2$}&2&21840\\
&&11&6720\\\hline
\multirow{2}{*}{$E_6(a_1)+A_1$}&\multirow{2}{*}{$\S_2$}&2&27840\\
&&11&19200\\\hline
$A_7$&.&.&17280\\\hline
$E_6$&.&.&13440\\\hline
\multirow{2}{*}{$D_7(a_2)$}&\multirow{2}{*}{$\S_2$}&2&38880\\
&&11&23760\\\hline
$D_6$&.&.&30240\\\hline
\multirow{2}{*}{$E_6(a_1)$}&\multirow{2}{*}{$\S_2$}&2&55680\\
&&11&38400\\\hline
\multirow{2}{*}{$D_5+A_2$}&\multirow{2}{*}{$\S_2$}&2&58500\\
&&11&1980\\\hline
\multirow{2}{*}{$D_6(a_1)+A_1$}&\multirow{2}{*}{$\S_2$}&2&109440\\
&&11&3600\\\hline
$A_6+A_1$&.&.&69120\\\hline
$A_6$&.&.&138240\\\hline
\multirow{2}{*}{$D_6(a_1)$}&\multirow{2}{*}{$\S_2$}&2&151200\\
&&11&60480\\\hline
\multirow{7}{*}{$2A_4$}&\multirow{7}{*}{$\S_5$}&5&135240\\
&&41&99246\\
&&32&52206\\
&&311&12516\\
&&221&7686\\
&&2111&126\\
&&11111&0\\\hline
$D_5+A_1$&.&.&181440\\\hline
\multirow{3}{*}{$A_5+A_2$}&\multirow{3}{*}{$\S_3$} &3&346080\\
&&21&198240\\
&&111&13440\\\hline
\multirow{2}{*}{$D_6(a_2)$}&\multirow{2}{*}{$\S_2$}&2&408240\\
&&11&257040\\\hline
\multirow{2}{*}{$A_5+2A_1$}&\multirow{2}{*}{$\S_2$}&2&383040\\
&&11&141120\\\hline
$(A_5+A_1)'$&.&.&483840\\\hline
$D_5(a_1)+A_2$&.&.&362880\\\hline
$D_5$&.&.&362880\\\hline
$A_4+A_3$&.&.&241920\\\hline
\multirow{2}{*}{$(A_5+A_1)''$}&\multirow{2}{*}{$\S_2$}&2&766080\\
&&11&282240\\\hline
\multirow{2}{*}{$D_4+A_2$}&\multirow{2}{*}{$\S_2$}&2&574560\\
&&11&30240\\\hline
$A_5$&.&.&967680\\\hline
$D_5(a_1)+A_1$&.&.&1088640\\\hline
$A_4+A_2+A_1$&.&.&483840\\\hline
$A_4+A_2$&.&.&967680\\\hline
\multirow{2}{*}{$A_4+2A_1$}&\multirow{2}{*}{$\S_2$}&2&1103760\\
&&11&347760\\\hline
\multirow{2}{*}{$D_5(a_1)$}&\multirow{2}{*}{$\S_2$}&2&1542240\\
&&11&635040\\\hline
\multirow{2}{*}{$A_4+A_1$}&\multirow{2}{*}{$\S_2$}&2&1753920\\
&&11&1149120\\\hline
$2A_3$&.&.&1209600\\\hline
$D_4+A_1$&.&.&1814400\\\hline
\multirow{2}{*}{$D_4(a_1)+A_2$}&\multirow{2}{*}{$\S_2$}&2&2116800\\
&&11&907200\\\hline
$A_3+A_2+A_1$&.&.&2419200\\\hline
\multirow{2}{*}{$A_4$}&\multirow{2}{*}{$\S_2$}&2&3507840\\
&&11&2298240\\\hline
\multirow{2}{*}{$A_3+A_2$}&\multirow{2}{*}{$\S_2$}&2&4687200\\
&&11&151200\\\hline
\multirow{3}{*}{$D_4(a_1)+A_1$}&\multirow{3}{*}{$\S_3$} &3&3864000\\
&&21&2486400\\
&&111&235200\\\hline
$A_3+2A_1$&.&.&7257600\\\hline
$D_4$&.&.&3628800\\\hline
$2A_2+2A_1$&.&.&4838400\\\hline
\multirow{3}{*}{$D_4(a_1)$}&\multirow{3}{*}{$\S_3$} &3&7728000\\
&&21&4972800\\
&&111&470400\\\hline
$A_3+A_1$&.&.&14515200\\\hline
$2A_2+A_1$&.&.&9676800\\\hline
\multirow{2}{*}{$2A_2$}&\multirow{2}{*}{$\S_2$}&2&13305600\\
&&11&6048000\\\hline
$A_2+3A_1$&.&.&14515200\\\hline
$A_3$&.&.&29030400\\\hline
$A_2+2A_1$&.&.&29030400\\\hline
\multirow{2}{*}{$A_2+A_1$}&\multirow{2}{*}{$\S_2$}&2&39916800\\
&&11&18144000\\\hline
$4A_1$&.&.&43545600\\\hline
\multirow{2}{*}{$A_2$}&\multirow{2}{*}{$\S_2$}&2&79833600\\
&&11&36288000\\\hline
$3A_1$&.&.&87091200\\\hline
$2A_1$&.&.&174182400\\\hline
$A_1$&.&.&348364800\\\hline
$A_0$&.&.&696729600\\\hline
\end{longtable}
\unskip
\unpenalty
\unpenalty}
\unvbox\ltmcbox
\medskip
\end{multicols}

\newpage
\begin{longtable}{|c|c|c|c|}
\caption{Type $F_4$} \label{table:f4} \\
\hline $N$ & $A_N$ & $\phi \in \widehat{A_N}$ &$\br{H^*(\B_N),\phi}$ \\ \hline \hline
 $F_4$ & . & . & 1 \\\hline
\multirow{2}{*}{ $F_4(a_1)$} & \multirow{2}{*}{$\S_2$}& 2 & 5 \\
&  & 11 & 2 \\\hline
\multirow{2}{*}{ $F_4(a_2)$ }&\multirow{2}{*}{$\S_2$} &2 & 14 \\
 & & 11 & 2 \\\hline
 $B_3$ & . & . & 24 \\\hline
 $C_3$ & . & . & 24 \\\hline
\multirow{5}{*}{ $F_4(a_3)$ }& \multirow{5}{*}{$\S_4$} & 4 & 42 \\
& & 31& 19 \\
& & 22& 10 \\
& & 211& 1 \\
& & 1111& 0 \\\hline
\multirow{2}{*}{ $C_3(a_1)$} & \multirow{2}{*}{$\S_2$}& 2 & 96 \\
 & & 11& 24 \\\hline
\multirow{2}{*}{ $B_2$} & \multirow{2}{*}{$\S_2$} & 2 & 96 \\
& & 11 & 48 \\\hline
 $\tilde{A}_2 + A_1$ & . & . & 96 \\\hline
 $A_2 + \tilde{A}_1$ & . & . & 96 \\\hline
 $\tilde{A}_2$ & . & . & 192 \\\hline
\multirow{2}{*}{ $A_2$} & \multirow{2}{*}{$\S_2$}& 2 & 168 \\
& & 11 & 24 \\\hline
 $A_1 + \tilde{A}_1$ & . & . & 288 \\\hline
\multirow{2}{*}{ $\tilde{A}_1$ }& \multirow{2}{*}{$\S_2$}& 2 & 432 \\
&  & 11 & 144 \\\hline
 $A_1$ & . & . & 576 \\\hline
 $A_0$ & . & . & 1152 \\\hline
\end{longtable}

\begin{longtable}{|c|c|c|c|}
\caption{Type $G_2$} \label{table:g2} \\
\hline $N$ & $A_N$ & $\phi \in \widehat{A_N}$ &$\br{H^*(\B_N),\phi}$ \\ \hline \hline
 $G_2$ & . & . & 1 \\\hline
\multirow{3}{*}{ $G_2(a_1)$} & \multirow{3}{*}{$\S_3$} & 3& 3 \\
 &  & 21& 1 \\
 &  & 111& 0 \\\hline
 $\tilde{A}_1$ & . & . & 6 \\\hline
 $A_1$ & . & . & 6 \\\hline
 $A_0$ & . & . & 12 \\\hline
\end{longtable}

\bibliographystyle{amsalphacopy}
\bibliography{induct}

\end{document}